\documentclass[twoside,11pt]{amsart}
\usepackage{amscd,amsfonts,amssymb}
\usepackage{graphicx}
\usepackage{color}

 \usepackage{epsfig}
\usepackage{float}

\newcommand{\R}{\mathbb R} 

\newcommand{\nwc}{\newcommand}
\nwc{\eps}{\varepsilon} \nwc{\bc}{\begin{center}}
\nwc{\ec}{\end{center}} \nwc{\bi}{\begin{itemize}}
\nwc{\ei}{\end{itemize}} \nwc{\disp}{\displaystyle}
\nwc{\pb}{\medskip\noindent} \nwc{\bt}{\begin{tabbing}}
\nwc{\et}{\end{tabbing}} \nwc{\be}{\begin{eqnarray*}}
\nwc{\ee}{\end{eqnarray*}} \nwc{\ba}{\begin{align}}
\nwc{\sleq}{\leqslant} \nwc{\bu}{\bar{u}} \nwc{\mF}{\mathbb{F}}

\def\be{\beta}
\def\Del{\Delta}

\def\Lam{\Lambda}
\def\lam{\lambda}

\newtheorem{thm}{Theorem}[section]
\newtheorem{lem}{Lemma}[section]

\theoremstyle{definition}

\theoremstyle{remark}

\numberwithin{equation}{section}

\author{Hailiang Liu}
\address{Department of Mathematics\\Iowa State University\\Ames, IA 50011.
} \curraddr{} \email{hliu@iastate.edu, Fax:(515)294-5454}

\author{Eitan Tadmor}
\address{Department of Mathematics\\
Center of Scientific Computation And Mathematical Modeling (CSCAMM)\\
and Institute for Physical Science and Technology (IPST)\\
University of Maryland\\
College Park, MD 20742.} \curraddr{}
\email{tadmor@cscamm.umd.edu}

\author{Dongming Wei}
\address{Department of Mathematics\\
University of Wisconsin-Madison\\
Madison, WI 53706} \curraddr{}
\email{dwei@math.wisc.edu}

\title[Global Regularity of the 4D Restricted Euler Equations]{Global Regularity of the 4D
Restricted Euler Equations}
\thanks{Work partially supported
by the NSF under grant \#DMS05-05975(H. Liu) and \#DMS07-07949 (E. Tadmor
and D. Wei).}
\date{\today}

\begin{document}

\begin{abstract}
We are concerned with the critical threshold phenomena in
the Restricted Euler (RE) equations. Using the spectral  and
trace dynamics we  identify  the critical thresholds for 3D and
the 4D restricted Euler equations. It is well known that the 3D RE
solutions blow up. Projected on the 3-sphere, the set of initial
eigenvalues which give rise to bounded stable solutions is reduced to a single point, which confirms that 3D RE blowup is generic. In contrast, we identify a surprisingly {\it rich} set of the initial spectrum on the 4-sphere which yields global smooth solutions; thus, 4D regularity is generic.
\end{abstract}


\maketitle

\centerline{\large\sf To Katepalli Sreenivasan on his 60th birthday, with friendship and admiration}

\section{Restricted Euler Equations and Spectral Dynamics}
We are concerned  with the questions of global regularity vs. finite-time breakdown of Eulerian flows governed by
$$
\partial_t u +u\cdot \nabla_x u= F, \quad x\in \R^n, \quad t>0.
$$
Here $u$ is the velocity field, $u:=(u_1,u_2,\ldots,u_n)^\top: \R^{1+n} \mapsto \R^n$, and its global  behavior is dictated by the different models of the forcing $F=F(u, \nabla u,\ldots)$.
For forcing involving viscosity and pressure, we meet the well
known Navier-Stokes (NS) equations,
\begin{equation}\label{2.3}
\partial_t u +u\cdot \nabla_x u= \nu\Delta u - \nabla p, \quad x\in \R^n, \quad t>0,
\end{equation}
augmented with the incompressibility condition, $\nabla\cdot u=0$ and subject to prescribed initial conditions
$u(x, 0)=u_0(x)$. In many applications, $\nu>0$ is  sufficiently small
 so that one can anticipate the behavior of slightly viscous NS solutions to be described by the Euler equations with   $\nu=0$ in (\ref{2.3}), at least for flows occupying the whole space so that the important effects of boundary layers can be ignored.

The velocity gradient  of the incompressible Euler
equations, $M:=\nabla_x u$ then solves
\begin{equation}\label{2.6}
\partial_t M  +  u\cdot \nabla_x M  +  M^2 =  -(\nabla \otimes \nabla )p.
\end{equation}
Taking the trace of (\ref{2.6}) while noting that $M$ is trace-free, $tr M=\nabla \cdot u=0$, one finds
$tr M^2 = - \Del p$ which dictates the pressure as $p =-\Del ^{-1} (tr M^2)$. The second term in
(\ref{2.6}) therefore amounts to the $n\times n$ time-dependent matrix
$$
 (\nabla \otimes \nabla )\Del ^{-1} (tr M^2)= R[trM^2].$$
Here $R[w]$ denotes the so called \emph{Riesz matrix} --- an $n\times n$ matrix whose entries,
$(R[w])_{jk}:=R_jR_k(w)$, involve the Riesz transforms $R_j$,
$R_j=-(-\Delta)^{-1/2}\partial_j$, i.e.,
$$
R[w]:= \left\{R_j R_k(w)\right\}_{j,k=1}^n, \quad \widehat{R_j(w)}(\xi)=-i\frac{\xi_j}{|\xi|}\hat {w}(\xi)\quad
{\rm for}\quad 1\leq j\leq n.
$$
This furnishes an equivalent, self-contained formulation of Euler equations, expressed in terms of the velocity gradient $M=\nabla_x u$, which is governed by,
\begin{align}\label{NSM}
\partial_t M  +  u\cdot \nabla_x M  +  M^2 = R[tr M^2], \quad M=\nabla_x u,
\end{align}
and subject to the trace-free initial data, $M(\cdot,0)=M_0$.
Observe that the invariance of incompressibility is already taken
into account in (\ref{NSM}) since $tr M^2 = tr R[tr M^2]$ implies $(\partial_t+u\cdot\nabla_x)tr M=0$ and hence
 $tr M=tr M_0=0$.

It is the \emph{global} nature  of the Riesz matrix,
$R[tr M^2]$, which makes the  issue of regularity for Euler and NS equations such an intricate question to solve, both analytically and numerically, \cite{Fef2000}. Various simplifications to this
pressure Hessian, $R[tr M^2]=-\nabla\otimes\nabla p$ were sought, e.g., \cite{Ler1934, Vie1982, Con1986, HMR1998,CFHOTW1999}.
In this paper we focus our attention on the so called
\emph{restricted Euler equations},  proposed in
\cite{Leo1975, Vie1982} as a \emph{localized} alternative  of the full Euler
equations (\ref{NSM}).  By the definition of the Riesz matrix, one has
$$
R[trM^2]=\nabla \otimes \nabla \Delta ^{-1}[trM^2]
  =\nabla \otimes \nabla  \int_{\R^n}K(x-y)(tr M^2)(y)dy,
$$
where  the kernel $K(\cdot)$  is given by
$$
K(x)=
  \begin{cases}
    {\displaystyle \frac{1}{2\pi}ln|x|} & \text{n=2}, \\
     {\displaystyle \frac{1}{(2-n)\omega_n|x|^{n-2}}} & \text{$n>2$},
  \end{cases}
$$
with $\omega_n$ denoting the surface area of the unit sphere in
$n$-dimensions. A direct computation yields
\begin{equation}\label{Rmat}
\partial_j\partial_k K*trM^2  =\frac{tr M^2}{n}\delta_{jk}
+ \int_{\R^n} \frac{|x-y|^2 \delta_{jk} - n(x_j -y_j)(x_k
-y_k)}{\omega_n |x-y|^{n+2}} tr M^2(y) dy.
\end{equation}

Ignoring the singular integrals on the right of (\ref{Rmat}), we are left with the
local part  of the Riesz matrix $R[tr M^2]$, given by $tr M^2 I_{n\times n}/n$.
We use this local term to approximate the  pressure Hessian in (\ref{NSM}).
The resulting Restricted Euler (RE) equations amount to
\begin{equation}\label{5.1}
\partial_t M  + u\cdot \nabla_x M +M^2= \frac{1}{n}tr M^2 I_{n\times n}.
\end{equation}
This is a matrix Ricatti equation for the $n\times n$ matrix $M$, which should mimic the dynamics of the
velocity gradient, $\nabla u$ in the full Euler equations. We observe that as in the full Euler equations,
incompressibility is  maintained in the restricted model,
since $tr M^2= tr[tr M^2 I_{n \times n}/n]$ implies $(\partial_t+u\cdot \nabla_x) tr
M=0$ and hence $tr M=tr M_0=0$. The 3D RE (\ref{5.1})  has attracted  great
attention since it was first introduced in \cite{Leo1975, Vie1982} as a local approximation to the full 3D Euler equations.
It can be used to understand the local topology of the Euler dynamics and to
capture certain statistical features of  physical turbulent flows, consult \cite{Can1992, CPS1999, Vie1982}.

What about the global regularity of the RE equations (\ref{5.1})?
the finite time breakdown of the three-dimensional restricted model
goes back to the original work of Viellefosse \cite{Vie1982}.
In \cite{LT2002a} we have shown that the 3D RE solutions break down at a finite time for \emph{all}
initial configurations $M_0$, except for the special case when $M_0$ has three real eigenvalues,
$$
\lambda_1(0)\leq \lambda_2(0)\leq \lambda_3(0), \quad \left\{\lambda_j(0)=\lambda_j(M_0)\right\}_{j=1}^3,
$$
which are aligned along the ray $(-r, -r, 2r), \ r\in\R^+$.
Thus, the finite time break down of the 3D RE equations is \emph{generic}.

In this paper we shall identify and compare between the restricted Euler equations
in 3D and 4D case, respectively.  To this end, we consider
a bounded, divergence-free, smooth vector field $u: \R^n \times
[0, T]\to \R^n$.  Let $x=x(\alpha, t)$ denote an orbit associated
to the flow by
$$
\frac{d x}{dt}=u(x, t), \quad 0<t<T, \quad x(\alpha, 0)=\alpha \in
\R^n.
$$
Then along this orbit, the velocity gradient tensor of the
restricted Euler equations (\ref{5.1}) satisfies
$$
\frac{d }{dt}M + M^2= \frac{tr M^2}{n} I_{n\times n}, \quad \frac{d}{dt}:=\partial_t +u\cdot \nabla_x.
$$
By the spectral dynamics lemma 3.1 in \cite{LT2002a},  the
corresponding eigenvalues  of $M$ satisfy
\begin{equation}\label{5.2}
\frac{d }{dt} \lam_i  +  \lam_i^2 =  \frac{1}{n}\sum_{j=1}^n \lam_j^2, \quad
i=1,\cdots, n.
\end{equation}
This is a closed system for $\Lam(t)=(\lam_1(t), \lam_2(t),
\ldots, \lam_n(t))$, which serves as a simple approximation for
the evolution of the velocity gradient field.

For arbitrary $n\geq 3$,  we use the spectral dynamics of $M$ in
order to show the existence of a large set of initial configurations
leading to finite time breakdown of (\ref{5.2}), generalizing
the previous result of \cite{Vie1982}. The finite time breakdown of the $n$-dimensional RE equations (and the  precise topology of the  breakdown) was established in \cite{LT2002a} after we identified  a set of $[n/2]+1$ global spectral invariants, interesting for their own sake.  Yet, this does not exclude the possible existence of other generic  sets of initial data, for which global smooth solutions exist. The distinction between these two sets of initial conditions is identified by the so-called \emph{critical threshold} surfaces in configuration space:  finite time breakdown occurs for super-critical initial data on ``one side" of the  such critical threshold, while the set of \emph{sub-critical} initial data on the ``other side" of the threshold yields global smooth solutions.

An interesting question  therefore arises, namely, whether there exists a critical threshold for the 4D restricted Euler equation. This remarkable critical threshold phenomena was identified in \cite{ELT2001, TW2008} for a class of essentially 1D Euler-Poisson equations, and in \cite{LT2002b} for a
1D convolution model for nonlinear conservation laws. The 2D
critical threshold phenomena has been recently confirmed for a
restricted Euler-Poisson system \cite{LT2003} and a rotating Euler
equation \cite{LT2004, CT2008}. In all these cases, we identified large, generic sets
of sub-critical initial data, which evolve to global smooth solutions.
This is in contrast to the generic scenario of finite-time blows up in the 3D RE equations\footnote{We should emphasize that generic sub-critical data are \emph{not} limited to a perturbative statement of global existence for initial data in the local neighborhood of certain ``preferred configurations". Instead, the precise notion of ``generic" sub-critical sets,  quantified below and the  references mentioned above, makes it clear the critical threshold phenomena we seek describes a global scenario in configuration space.}.

In this paper we identify the exact critical thresholds for the 4D
restricted Eulerian (RE) equations and we conclude with the
surprising result that in the 4D case, the RE equations admit a
large, generic set of sub-critical initial data which give rise to
global smooth solutions.

A summary of our results is outlined below. We say that $\Lam_0 \in
\R^n$ is {\it sub-critical} if there exists a global
solution in time of (\ref{5.2}), subject to initial conditions,
$\Lam(0)=\Lam_0$. A first observation rests on the obvious
symmetries of (\ref{5.2}).

\begin{lem}\label{lem:scale}
If $\Lam$ is sub-critical then so is $r\Lam, \ \forall r>0$.
Moreover, $\Lam_\sigma=\{\lam_{\sigma(j)}, \ \forall \sigma\in
\pi_n\}$ is also subcritical.
\end{lem}

For the proof we note that if $\Lam(t)$ is the global solution
corresponding to $\Lam_0$, then $r\Lam(rt), \ r>0$ is the global
solution corresponding to $r\Lam_0$. Also, equations (\ref{5.2})
remain invariant under arbitrary permutation $\sigma$ which amounts
to reordering, exchanging  the $\lam_j$-equation with
$\lam_{\sigma(j)}$-equation. It follows that the set of sub-critical
initial data consists of rays, and therefore, it is enough to
consider the projection of this set on the unit sphere. In fact, we
can restrict attention to an orthant of any convex set containing the
origin. In this context we have

\begin{thm}\label{thm:3DRE}
Solutions to (\ref{5.2}) with $n=3$ remain bounded for all time if and
only if the initial data $\Lambda_0:=(\lambda_{10}, \lambda_{20},
\lambda_{30})$ lie in the following set
$$
r\{ (-1, -1, 2)_\sigma \}.
$$
\end{thm}

Restricted to one orthant of the unit sphere, we thus find that the
3D RE equations admit only one sub-critical point. In this sense,
the finite-time breakdown of 3D RE is generic.  This result was
already obtained in \cite{LT2002a} by spectral dynamics analysis. In \S
 4 we present an alternative, equivalent argument based on trace dynamics of $tr(M^k)$, $1\leq k\leq n$,  which paves the way for identifying our 4D critical threshold surface in \S 5.

In contrast to this generic 3D finite-time breakdown, the 4D RE equations admit a large class of global smooth solutions. Our 4D results are summarized below.

\begin{thm}\label{thm:4DRE-}
Solutions to (\ref{5.2}) with $n=4$ remain bounded for all time if
and only if the initial data $\Lambda_0:=(\lam_{10}, \lam_{20},
\lam_{30}, \lam_{40})$ with $\sum_{j=1}^4 \lam_{j0}=0$, up to a permutation, lie in one of the
following sets

\{i\}Two pairs of arbitrary complex eigenvalues,
$$
\Lam_0 \in \{(a+ib, a-ib, -a+ic, -a-ic), \quad  bc \neq 0 \}.
$$

\{ii\}One pair of complex eigenvalue with two equal
real eigenvalues
$$
\Lam_0 \in \{(a+ib, a-ib, -a, -a), \quad  b \not=0 \}.
$$

\{iii\} Real eigenvalues,
$$
\Lam_0 \in \{(a+b, a-b, -a, -a), \quad b \in
[-2a, 2a], \; a\geq 0\}.$$

\end{thm}

Expressed in terms of traces $m_k:= tr(M^k)$, these  initial configurations  form a ``large" sub-critical set which can be realized by  its projection on the surface $\Sigma$,
\[
\Sigma:=\{\Lam \ \big| \ 4m_4-2m_2^2- 2m_2 + 3=0, \quad m_1=0\},
\quad m_k:=\sum_{k=1}^4 \lam^k_j.
\]
\begin{thm}\label{thm:4DRE}
Solutions to (\ref{5.2}) with $n=4$ remain bounded for all time if
and only if  there exists a $r>0$ such that the initial data
$\Lambda_0:=(\lam_{10}, \lam_{20}, \lam_{30}, \lam_{40})$ lie in the
following set
\[
r \Lam_0 \in \Sigma  \cap \left[ \{|m_3| \leq 1.5(1 -m_2), m_2\leq 1
\} \cup \{m_3=1.5(m_2-1), \quad m_2 > 1\} \right].
 \]
\end{thm}
The set stated in Theorem \ref{thm:4DRE} is non-trivial; in fact, it
contains non-zero neighborhoods.  We note that a recent study in \cite{4d} on 4D incompressible Navier-Stokes equations suggests that the energy transfer in 4D is indeed more efficient than in 3D.




After this introduction of restricted Euler equations and the
associated spectral dynamics, we identify the 4D sub-critical initial
configurations in terms of eigenvalues in \S 2. An alternative formulation of the
spectral dynamics --- called trace dynamics is derived in \S 3. Based
on the trace dynamics we identify the critical thresholds for 3D
case in \S 4 and the 4D model in \S 5. Finally in the appendix we establish the
correspondence between the sub-critical eigenvalues and the sub-critical set
for initial traces.

\section{4D spectral dynamics}
Let  $\lambda=\lambda_i$ solve the restricted Euler equation
\begin{equation}\label{6.1-}
\frac{d}{dt}\lambda_i+ \lambda_i^2= \frac{1}{4} \sum_{j=1}^4
\lambda_j^2,  \quad
i=1\cdots 4.
\end{equation}
Two independent global invariants obtained in \cite{LT2002a} are
\begin{subequations}\label{eqs:gl12}
\begin{equation}\label{gl1}
(\lambda_1-\lambda_2)(\lambda_3-\lambda_4)=(\lambda_{10}-\lambda_{20})(\lambda_{30}-\lambda_{40})
\end{equation}
and
\begin{equation}\label{gl2}
(\lambda_1-\lambda_3)(\lambda_2-\lambda_4)=(\lambda_{10}-\lambda_{30})(\lambda_{20}-\lambda_{40})
\end{equation}
\end{subequations}
We now prove Theorem \ref{thm:4DRE-} based on these global invariants. In view of the incompressibility invariant $\sum_{j=1}^4\lambda_j=0$, we can express  the remaining three spectral degrees of freedom as
$\Lambda=(a+b, a-b, -a+c, -a-c)^\top$, where $a$ is real, $a\in \R$,  and $b$ or $c$ are either real $b,c\in \R$ or purely imaginary, $b,c \in i\R$.
The two global invariants  (\ref{eqs:gl12}) now read
\begin{subequations}\label{eqs:abc}
\begin{equation}\label{bc}
bc = b_0c_0
\end{equation}
and
\begin{equation}\label{abc}
4a^2-b^2-c^2=4a_0^2-b_0^2-c_0^2.
\end{equation}
\end{subequations}
The spectral dynamics (\ref{6.1-}) amounts to the $3\times 3$ closed system,
\begin{subequations}\label{eqs:a}
\begin{align}\label{a}
\frac{d}{dt} a & = -\frac{1}{2}b^2 +\frac{1}{2}c^2,\\ \label{b}
\frac{d}{dt} b&=-2ab,\\ \label{c}
\frac{d}{dt}c &=-2ac,
\end{align}
\end{subequations}
subject to initial data $(a_0, b_0, c_0)$.
Observe that both $b=0$ and $c=0$ are global invariants, thus the only equilibrium points of (\ref{eqs:a}) when $b_0c_0 \not=0$ lie along the curves $(0,b^*,c^*), \ b^*=\pm c^*$.
From (\ref{b}) and (\ref{c}) it is clear that if either $b_0$ or $c_0$ are purely imaginary then they remain so for all time. Thus, we need to  discuss three cases in order.

\{i\} Two pairs of complex eigenvalues, $a_0\pm ib_0$ and $-a_0\pm ic_0$
with $b_0c_0\neq 0$. Setting  $(a, b, c) \mapsto (a, ib, ic)$ in (\ref{abc}) we obtain the global invariant
$$
4a^2+b^2+c^2=4a_0^2+b_0^2+c_0^2.
$$
In this case, \emph{all} trajectories remain bounded for all time. \\

{\{ii\} One pair of complex eigenvalues, $a_0 \pm ib_0, \ b_0 \neq 0$ and $-a_0\pm c_0$. Setting  $(a, b, c) \mapsto (a, ib, c)$ in (\ref{abc}), then the global invariant (\ref{abc}) becomes
\[
4a^2+b^2-c^2=4a_0^2+b_0^2-c_0^2.
\]
We distinguish between two cases. If $c_0=0$, then by (\ref{c}) $c(t)\equiv 0$, and the reduced global invariant, $4a^2+b^2=4a_0^2+b_0^2$, implies that both $a$ and $b$ remain bounded for all $t>0$.  
If $c_0 \not =0$, then the equation (\ref{a}) becomes
$$
\frac{d}{dt} a=\frac{1}{2} \left(b^2 +\frac{(b_0c_0)^2}{b^2} \right),
$$ 
this shows that no finite equilibrium point of the system (\ref{a})-(\ref{c}) is stable, which excludes the possibility of
a globally bounded solution when $b_0c_0\not=0$. 

{\{iii\} Two real eigenvalues, $a_0\pm b_0$ and $-a_0\pm c_0$. Again, we distinguish between two cases. Assume that two initial eigenvalues coincide, say $c_0=0$ (if $b_0= 0$, we end up with a similar scenario which amounts to a permutation of the $c_0=0$ case). Then $c(t)\equiv 0$ and the  remaining $(a, b)$ satisfy the reduced $2\times 2$ system
\[
\frac{d}{dt} \left(\begin{array}{c}a\\b\end{array}\right)
= \left(\begin{array}{c}-b^2/2\\ -2ab\end{array}\right),
\]
with the corresponding global invariant  $4a^2-b^2=4a_0^2-b_0^2$. Now,  since $\frac{d}{dt}a=-b^2/2\leq 0$, it follows that  $a>0$ is decreasing  while $b$ must approach the stable equilibrium points $(a^*>0,0)$ along the positive $a$-axis, as $\frac{d}{dt}b=-2ab$ has the opposite sign of $b$. Thus, trajectories remain bounded in the invariant  sector $|b_0|\leq 2a_0$.

Finally, assume no pair of initial eigenvalues coincide, $b_0c_0 \not=0$. Then, since the global invariants (\ref{bc}) and (\ref{abc}) are not compact, the only possible bounded solutions are those converging to the equilibrium points $(0, \pm c^*, c^*)$. But when substituted into both (\ref{bc}) and (\ref{abc}), this implies
\[
4a_0^2=(b_0 \pm c_0)^2,
\]
which is satisfied only if at least one pair of initial eigenvalues coincide, i.e. $a_0 \pm b_0=-a_0 \pm c_0$.
We conclude that for real eigenvalues, only those of the form $(a_0+b_0, a_0-b_0, -a_0, -a_0)$ with $|b_0|\leq 2a_0, a_0\geq 0$ lead to global bounded solution.

\section{Trace dynamics}
This section is devoted to an alternative formulation of the
spectral dynamics in terms  of real quantities $m_k:=\sum_{j=1}^n
\lambda_j^k, \quad k=1, \cdots, n$, where $\lambda=\lambda_i$
solves the restricted Euler equation
\begin{equation}\label{6.1}
\frac{d}{dt}\lambda_i+ \lambda_i^2= \frac{1}{n} \sum_{j=1}^n
\lambda_j^2,  \quad
i=1\cdots n.
\end{equation}
This is motivated by the trace dynamics originally studied in
\cite{Vie1982} for $n=3$. The use of trace dynamics enables us to obtain an explicit description of the critical threshold surface for initial configurations.

Here we seek an extension for the general $n$-dimensional setting,
which is summarized in the following

\begin{lem}\cite{LT2002a} \label{trace}
Consider the $n$-dimensional  restricted Euler system (\ref{6.1})
subject to the incompressibility condition $m_1:=\sum_{j=1}^n
\lambda_j = 0$. Then the traces $m_k$ for $k=2,  \cdots, n$
satisfy a closed dynamical system, see (\ref{9.2})-(\ref{9.5}) below, which governs the local topology of the
restricted  flow.
\end{lem}
\begin{proof}
Based on the spectral dynamics  the evolution  equation for each
eigenvalue $\lambda_i$ can be written as
$$
\frac{d}{dt} \lambda_i +\lambda_i^2=\frac{1}{n}m_2, \quad
i=1\cdots n.
$$
By multiplying $k\lambda_i^{k-1}$ and summation over $i$ we obtain
$$
\frac{d}{dt} m_k + k m_{k+1} = \frac{k}{n}m_2m_{k-1}, \quad
k=2\cdots n.
$$
Note that $m_1=0$ we have
\begin{align}\label{9.2}
&  \frac{d}{dt} m_2+ 2m_3=0,\\\label{9.3}
& \frac{d}{dt} m_3  + 3m_4  = \frac{3}{n}  m_2^2,\\
& \cdots \notag \\
\label{n+1} & \frac{d}{dt} m_n  +nm_{n+1}= m_{n-1}m_2.
\end{align}
To close the system, it remains to express $m_{n+1}$ in terms of
$(m_1, \cdots, m_n)$. To this end  we  utilize the characteristic
polynomial
\begin{equation}\label{CH}
\lambda_j^n  +  q_1\lambda_j^{n-1} + \cdots q_{n-1} \lambda_j +
q_n = 0,
\end{equation}
expressed in terms of the characteristic coefficients
$$
q_1 = -m_1 =0,  \quad q_2=-\frac{1}{2}m_2,  \quad q_3=-m_3/3,\quad
q_4=-m_4/4 +m_2^2/8, \quad \cdots.
$$
Note that the $q$'s can be expressed in terms of $(m_1, \cdots,
m_n)$. Using  (\ref{CH}) one may reduce $m_{n+1}$ in (\ref{n+1})
to lower-order products. In fact, $\sum_{j=1}^n (\lambda_j \times
(\ref{CH})_j)$ gives
\begin{equation}\label{9.5}
 m_{n+1} + q_2m_{n-1} +\cdots +q_{n-1}m_2=0.
\end{equation}
Substitution into (\ref{n+1}) yields the closed  system we sought
for.
\end{proof}

We demonstrate the above procedure by considering the  two examples
of 3D and 4D critical thresholds.

\section{3D critical thresholds: finite time blowup}
This section is devoted to the study of the 3D critical thresholds,
see \cite{Vie1982, Can1992}. In the three dimensional case one has
$$
q_1 = 0, \quad  q_2=-\frac{1}{2}m_2, \quad q_3=\prod_{j=1}^n
\lambda_j =-\frac{1}{3}m_3,
$$
hence
$$
\lambda_i^3 -\frac{1}{2}m_2 \lambda_i -\frac{1}{3}m_3=0, \quad
i=1, 2, 3.
$$
Multiplying by $\lambda_i$ and taking the summation over $i$ we
find
$$
m_4=\frac{1}{2}m_2^2.
$$
Thus a  closed system is obtained,
\begin{align}\label{3_1}
& \frac{d}{dt} m_2 + 2m_3=0,  \\ \label{3_2} &  \frac{d}{dt} m_3  +
\frac{1}{2} m_2^2=0.
\end{align}
From (\ref{3_1})-(\ref{3_2}) it follows that
$$
\frac{d}{dt}[6m_3^2-m_2^3]=6m_3 \frac{d}{dt} m_3 -3m_2^2 \frac{d}{dt} m_2 =0,
$$
which yields a global invariant
$$
6m_3^2- m_2^3={\rm Const.}
$$
We consider the phase plane $(m_2, m_3)$, except for the
separatrix $6m_3^2=m_2^3$, all other solutions would not approach
the origin. The phase plane is divided into two parts by this
separatrix. The nonlinearity ensures that trajectories which do
not pass the origin must lead to infinity at finite time. In fact
for initial data from the region $\{(m_2, m_3), \quad m_2
> \sqrt[3]{6} \; m_3^{2/3}\}$, the corresponding trajectories will
remain in this region since the system (\ref{3_1}), (\ref{3_2}) is
autonomous. Therefore (\ref{3_2}) leads to
\begin{equation}\label{m33}
\frac{d}{dt} m_3 <-\frac{1}{2}\sqrt[3]{36}\;  m_3^{4/3}.
\end{equation}
Since $\frac{d}{dt} m_3=-\frac{1}{2}m_2^2$, $m_3(t)$ is always decreasing
in time. Even for positive $m_3(0)$, there exists a finite time
$T^*$ such that $m_3(T*)<0$. The integration of (\ref{m33}) over
$[T^*, t)$ gives
$$
m_3(t)<\left[\frac{3}{2}\sqrt[3]{36}(t-T^*) +m_3(T^*)^{-1/3}
\right]^{-3}.
$$
This shows that $m_3(t) \to -\infty$ when $t$ approaches a time
before
$$
T^* +\frac{2}{3\sqrt[3]{36}}(-m_3(T^*))^{-1/3}.
$$
Finite time breakdown can be similarly justified for initial data
lying in the region  $\{(m_2, m_3), \quad m_2 < \sqrt[3]{6} \;
m_3^{2/3}\}$. These facts enable us to conclude the following
\begin{thm} Consider the system (\ref{3_1})-(\ref{3_2}) with initial
data $(m_2(0), m_3(0))$. The global bounded solution exists if and
only if the initial data lie on the curve
$$
\left\{ (m_2, m_3) \Big| \quad   m_3=\frac{1}{\sqrt{6}}m_2^{3/2}
\right\}.
$$
\end{thm}

We now turn to interpret this condition in terms of the
eigenvalues. Set $\Lambda=(\lambda_1, \lambda_2, \lambda_3)$, the
above critical stable set can be written as
$$
\Omega=\Big\{\Lambda \Big|\quad \sum_{k=1}^3 \lambda_k ^3=
\frac{1}{\sqrt{6}} \left( \sum_{k=1}^3 \lambda_k ^2 \right)^{3/2},
\quad \sum_{k=1}^3 \lambda_k =0 \Big\}.
$$
The homogeneity of the above constraint in terms of eigenvalues
implies that if $\Lambda \in \Omega$, then $r \Lambda \in
\Omega \quad \forall r>0$.

Without loss of generality we consider the restriction of $\Omega$
onto a ball $\sum_{k=1}^3 \lambda_k ^2=r^2$, denoted by
$\Omega(r)$. There are two cases to be considered:

The initial eigenvalues contain complex components, say
$\Lambda_0=(a-bi, a+bi, c)$ for real $a, b, c \in \R$. The
restricted set $\Omega(\sqrt{6})$ is determined by
$$
c+2a=0, \quad 2a^2-2b^2+c^2=6, \quad
2a(a^2-3b^2)+c^3=r^3/\sqrt{6}=6.
$$
Eliminating $c$ we have
$$
6a^2-2b^2=6, \quad -6a(a^2+b^2)=6 \Rightarrow 4a^3-3a+1=(a+1)(2a-1)^2=0,
$$
which has real roots  $a\in \{-1, 0.5, 0.5\}$, from which no real
$b\not= 0$ can be found.

The only possible scenario is the real eigenvalue $\Lambda_0=(a, b,
c)\in \R^3$. Restriction again on $\Omega(\sqrt{6})$ we have
$$
a+b+c=0, \quad a^2+b^2+c^2=6, \quad a^3+b^3+c^3=r^3/\sqrt{6}=6.
$$
Eliminating $a, b$ we have $c^3-3c-2=0$ with real roots $c\in \{2,
-1, -1\}$. The symmetric property implies that $a, b$ also lie in
the set $\{2, -1, -1\}$. In short one has
$$
\Omega(\sqrt{6})=\{\Lambda| (-1, -1, 2), (-1, 2, -1), (2, -1,
-1)\}.
$$
This when combined with the above scaling property leads to the
result stated in Theorem 1.1.

\section{4D critical thresholds: global regularity}

In the four dimensional case one has
$$
q_1 = 0, \quad  q_2=-\frac{1}{2}m_2, \quad q_3=
-\frac{1}{3}m_3,\quad q_4= -\frac{m_4}{4}  + \frac{m_2^2}{8}.
$$
Hence
$$
\lambda_i^4 -\frac{1}{2}m_2 \lambda_i^2  -\frac{1}{3}m_3\lambda_i
-\frac{m_4}{4} +\frac{m_2^2}{8} =0, \quad i=1\cdots 4.
$$
Multiplying by $\lambda_i$ and taking the summation we obtain
$$
m_5=\frac{1}{2}m_2m_3 + \frac{1}{3}m_3m_2=\frac{5}{6}m_2m_3.
$$
Therefore the resulting closed system becomes
\begin{align}\label{m1}
& \frac{d}{dt} m_2= - 2m_3,   \\ \label{m2} & \frac{d}{dt} m_3=
\frac{3}{4}  m_2^2 -3m_4, \\ \label{m3}
 & \frac{d}{dt} m_4= -\frac{7}{3}m_3m_2.
\end{align}
From (\ref{m1}) and (\ref{m3}) it follows that
$$
\frac{d}{dt}\left( m_4 -\frac{7}{12}m_2^2 \right)=0 ,
$$
which gives a global invariant
which gives the global invariant
\begin{equation}\label{eq:1stvar}
m_4- \frac{7}{12}m_2^2 =m_{40}-\frac{7}{12}m_{20}^2.
\end{equation}
Substitution of this into (\ref{m2}) leads to
$$
\frac{d}{dt}{m}_3=-m_2^2 -3(m_4-\frac{7}{12}m_2^2).
$$
In order to ensure global bounded solution (excluding globally decreasing $m_3$) it is necessary to
consider trajectories for which
\begin{equation}\label{aa}
m_4(t)- \frac{7}{12}m_2^2(t)=-\frac{l^2}{3}, \quad l>0.
\end{equation}
We thus have a closed system for
$(m_2, m_3)$
\begin{equation}\label{m23}
\frac{d}{dt}{m}_2=-2m_3, \quad \frac{d}{dt}{m}_3=-m_2^2+l^2
\end{equation}
with a moving parameter $l$ determined by (\ref{aa}) with $t=0$.
This system has two critical points $(-l, 0)$ and $(l, 0)$; it is
easy to verify that  as equilibrium points of system (\ref{m23}),
$(-l, 0)$ is a spiral and $(l, 0)$ is a saddle for the
corresponding linearized system.

This structure suggests that part of  separatrix' of this system
may serve as the critical threshold. Note that
$$
\frac{d}{dt}\left(3m_3^2-m_2^3+3l^2m_2 \right)
=6m_3\frac{d}{dt}{m}_3-3m_2^2\frac{d}{dt}{m}_2+3l^2\frac{d}{dt}{m}_2=0.
$$
Thus the 2nd global invariant when passing $(m_2, m_3)=(l, 0)$ becomes
$$
3m_3^2-m_2^3 +3l^2m_2=2l^3,
$$
yielding two separatrixes
\begin{equation}\label{eq:2ndvar}
3m_3^2=m_2^3 -3l^2m_2 +2l^3=(m_2+2l)(m_2-l)^2.
\end{equation}
We note in passing that the two  invariants (\ref{eq:1stvar}),(\ref{eq:2ndvar}) are
in fact the same spectral invariants we had before in (\ref{eqs:abc}),
which are now reformulated in terms of the traces $m_2$ and $m_3$.
Thus, for example, the straightforward identity
\[
12 m_4- 7m_2^2 \equiv -4(4a^2-b^2-c^2)^2 -48(bc)^2,
\]
reveals the relation betwen the trace-based invariant (\ref{eq:1stvar}) and the spectral invariant (\ref{abc}).

In the phase plane $(m_2, m_3)$, this consists of a closed curve
for $-2l\leq m_2\leq l$ and two open branches for $m_2>l$.  The
phase plane analysis suggests that the global bounded solution
exists if and only if  the initial data satisfy (\ref{aa}) and
$$
(m_2, m_3)(0)\in \Gamma,
$$
where
\begin{align*}
\Gamma:=\Big\{ (m_2, m_3)| \quad |m_3| & \leq
\frac{l-m_2}{\sqrt{3}}\sqrt{m_2+2l}, \quad -2l\leq m_2\leq l
\Big\}
\\
&  \cup \left\{ m_3=\frac{m_2-l}{\sqrt{3}}\sqrt{m_2+2l},  \quad
m_2>l \right\}
\end{align*}
and the moving parameter $l$ is determined by (\ref{aa}). Also we
can show that if initial data do not belong to $\Gamma$ the
solution becomes unbounded in finite time.

\begin{figure}[htbp]\label{fig:sep}
\begin{center}
\includegraphics[height=0.35\textheight,width=0.8\textwidth]{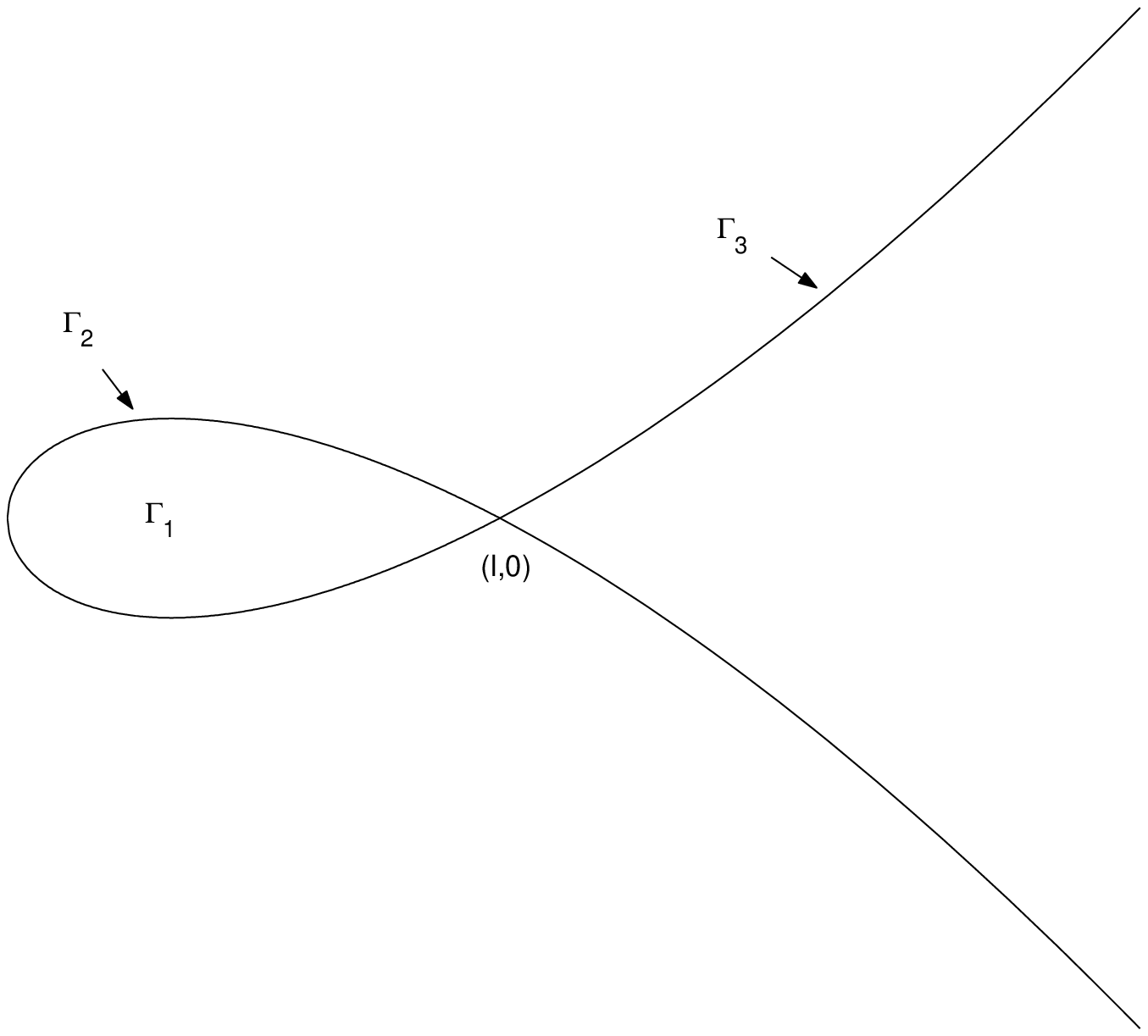}
\end{center}
\end{figure}


This figure depicts trajectories for system (\ref{m23}) in $(m_2, m_3)$ plane, which
contain the boundary of the non-trivial set $\Gamma$ with
(\ref{aa}).

From above analysis it follows that the
solutions remain bounded for all time if and only if
the initial data $\Lambda_0$ lies in the following set
\begin{equation}\label{sg}
\Lambda_0 \in \cup_{l> 0} \{S\cap\Gamma\},
\end{equation}
where $$\Gamma:=\Gamma_1 \cup \Gamma_2 \cup \Gamma_3,$$
$$S:=\{\Lambda \; | \;
m_4-\frac{7}{12}m_2^2=-\frac{l^2}{3}\},$$
$$\Gamma_1:=\{(m_2,m_3) \; | \; |m_3|<
\frac{l-m_2}{\sqrt{3}}\sqrt{m_2+2l},\; -2l\leq m_2<l \},$$
$$\Gamma_2:=\{(m_2,m_3) \; | \; |m_3|=
\frac{l-m_2}{\sqrt{3}}\sqrt{m_2+2l},\; -2l\leq m_2 < l\},$$
$$\Gamma_3:=\{(m_2,m_3) \; | \; m_3=
\frac{m_2-l}{\sqrt{3}}\sqrt{m_2+2l},\;  m_2\geq l\}.$$\\

We now turn to Theorem \ref{thm:4DRE}.
Let $r>0$ be a moving parameter, we restrict attention to the parameterized surface
$m_2+2l=3r^2$. Clearly the constraint $m_2\geq -2l$ is ensured for
any real $r$. For any $r>0$, the set $S$ restricted on this surface
is represented as
$$
 4m_4-2m_2^2- 2  m_2r^2 + 3r^4=0.
$$
This is a parabolic cylinder in the space $(m_2, m_3, m_4)$.
Applying the scaling property stated in Lemma 1.1, we may set $r=1$,
and denote the set $S$ with constraint $m_2+2l=3$ and $m_1=0$ as
\[
\Sigma:=\{\Lam \ \big| \ 4m_4-2m_2^2- 2  m_2 + 3=0, \quad m_1=0\},
\quad m_k:=\sum_{k=1}^4 \lam^k_j.
\]
The first half of the set $\Gamma|_{\Sigma}$ is supported where $-2l
\leq m_2 \leq l$, together with $m_2+2l=3$, i.e., $l=\frac{3-m_2}{2}$, leading to $m_2\leq 1 $.
In this case, the restriction
$$
|m_3| \leq {l-m_2}=\frac{3-m_2}{2}-m_2=1.5(1-m_2),
$$ yields
\[
\Omega_1= \big\{(m_2,m_3) \ \big| \ |m_3| \leq 1.5(1-m_2), \quad m_2
\leq 1 \big\}.
\]
For reals $\Lambda$, $m_2 \geq 0$; the fact of no lower bound for
$m_2$ suggests that any complex eigenvalue with zero divergence may
well lie in $\Sigma \cap \Omega_1$.

The second $\Gamma|_{\Sigma}$-constraint, supported on $m_3=m_2-l$
requiring $m_2> l=\frac{3-m_2}{2}$, i.e., $m_2 >1 $ leading to
\[
\Omega_2 = \big\{ (m_2,m_3) \ \big| \ m_3=1.5(m_2-1),  \quad m_2
> 1 \big\}.
\]
The above set $\Sigma \cap [\Omega_1 \cup\Omega_2]$ is `fat'. Note the
3D case is similar to the special case $l=0$ which restricts to a
large subcritical set. \\

\section{Appendix}
 Finally we turn to interpretation of eigenvalues presented in Theorem \ref{thm:4DRE-} in terms of the subcritical sets in Theorem \ref{thm:4DRE}, or the equivalent set (\ref{sg}).
\vskip 0.5cm
\noindent{\{i\} Two pairs of complex eigenvalues.} The eigenvalues must be
$\lambda_1=a+bi,\lambda_2=a-bi,\lambda_3=-a+ci,\lambda_4=-a-ci$,
where $a,b,c\in \mathbb{R}$ and $b\neq 0$, $c\neq 0.$ A direct calculation gives
$$ m_2=4a^2-2b^2-2c^2, $$
$$ m_3=6a(c^2-b^2), $$
$$ m_4=4a^4+2b^4+2c^4-12a^2(b^2+c^2), $$
$$ m_2^2-l^2=3(m_4-\frac14m_2^2)=3(b^2-c^2)^2-24a^2(b^2+c^2),$$
$$ l^2=(4a^2+b^2+c^2)^2+12b^2c^2.$$
It follows that $-2l \leq m_2 \leq l$ and $l>4a^2+b^2+c^2.$ Then
$$\frac{l-m_2}{\sqrt{3}}\sqrt{m_2+2l}>\frac{3(b^2+c^2)}{\sqrt{3}}\sqrt{12a^2}=6|a|(b^2+c^2)>|m_3|.$$
Thus we know $S \cap \Gamma=S\cap\Gamma_1.$

\vskip 0.5cm
\noindent{\{ii\} One pair of complex eigenvalues and two real
eigenvalues.} The four eigenvalues must be
$\lambda_1=a+bi,\lambda_2=a-bi,\lambda_3=-a+c,\lambda_4=-a-c$,
where $a,b,c\in \mathbb{R}$ and $b\neq 0.$ Changing $c$ in Case II to $-ci$ we immediately obtain
$$ m_2=4a^2-2b^2+2c^2, $$
$$ m_3=-6a(b^2+c^2), $$
$$ m_4=4a^4+2b^4+2c^4-12a^2(b^2-c^2), $$
$$ m_2^2-l^2=3(m_4-\frac14m_2^2)=3(b^2+c^2)^2-24a^2(b^2-c^2),$$
$$ l^2=(4a^2+b^2-c^2)^2-12b^2c^2.$$
Suppose $-2l\leq m_2 \leq l$ for $l>0$, we distinguish two cases:  \\
(1) If $4a^2+b^2-c^2\geq 0$, then $l\leq 4a^2+b^2-c^2$, and
$$\frac{l-m_2}{\sqrt{3}}\sqrt{m_2+2l}\leq \frac{3(b^2-c^2)}{\sqrt{3}}\sqrt{12a^2}=6|a(b^2-c^2)|\leq|m_3|.$$
It becomes an equality if and only if $c=0.$ \\
(2) If $4a^2+b^2-c^2<0,$ then $l<c^2-4a^2-b^2,$ and
$l-m_2=b^2-c^2-8a^2<0.$ It's a contradiction with $m_2<l.$ For $m_2>l$,
the constraint  $m_3=\frac{m_2-l}{\sqrt{3}}\sqrt{m_2+2l}$ leads to the relation (\ref{p0}), i.e.,
$$ p=\left(3m_3^2-m_2^3-9m_2\left(m_4-\frac{7}{12}m_2^2\right)\right)^2+108\left(m_4-\frac{7}{12}m_2^2\right)^3=0.$$
Calculation shows that $p=432b^2c^2(b^2+(2a-c)^2)^2(b^2+(2a+c)^2)^2.$ So $p=0$ if and only if $c=0.$
Thus we know that $S\cap\Gamma=S\cap \Gamma_2$, and two real eigenvalues must be equal.
\vskip 0.5cm
\noindent{\{iii\} all the eigenvalues are real}.
Suppose the four real eigenvalues are $a,b,c$ and $-(a+b+c)$, $a,b,c\in
\mathbb{R}$, then $m_2\geq 0$.  From the set $S$ in (\ref{sg}) it follows
$$
\frac{1}{3}( m_2^2-l^2)=m_4-\frac14m_2^2\geq 0,
$$
here we have used the inequality $(\alpha+\beta+\gamma+\delta)^2\leq 4
(\alpha^2+\beta^2+\gamma^2+\delta^2)$. These together lead to  $m_2 \geq l$.
Thus if all the eigenvalues are real, then $S\cap \Gamma=S\cap
\Gamma_3.$ \\
Because of the homogeneousness, we can assume the four real eigenvalues are $ 1+s,-1+w,-1$ and $1-s-w$(if $\Lambda_0\neq 0$). Let's do the following calculation.\\
From  $m_3=\frac{m_2-l}{\sqrt{3}}\sqrt{m_2+2l}$ it follows that
$$ 3m_3^2=(m_2-l)^2(m_2+2l)=m_2^3-3m_2l^2+2l^3,$$
$$ \Rightarrow [3m_3^2-m_2^3+3m_2l^2]^2= 4(l^2)^3.$$
Using $l^2=-3(m_4-\frac{7}{12}m_2^2)$ we have
\begin{equation}\label{p0}
     p:=\left(3m_3^2-m_2^3-9m_2\left(m_4-\frac{7}{12}m_2^2\right)\right)^2+108
     \left(m_4-\frac{7}{12}m_2^2\right)^3=0.
\end{equation}
Calculation shows that
$$p=-27(s+2)^2w^2(s-w+2)^2(2s+w)^2(s+2w-2)^2(s+w-2)^2.$$
So $p=0$ if and only if
$s=-2$ or $w=0$ or $s-w+2=0$ or $2s+w=0$ or $s+2w-2=0$ or $s+w-2=0.$ They are all the same if we consider the
homogeneousness and permutation. Now we know that the four eigenvalues must be in the form $r(1+s,-1,-1,1-s).$
We claim that the range for $s$ is $[-2,2].$ \\
(i)For $-2 \leq s \leq 2$, it's easy to check that $\Lambda \in S\cap \Gamma_3.$\\
(ii)For $s>2$ or $s<-2$, we calculate $l,m_2,m_3$ to obtain
$$ l={s}^2-4,\quad m_2=4+2{s}^2,\quad m_3=6{s}^2. $$
So $$\frac{m_2-l}{\sqrt{3}}\sqrt{m_2+2l}=\frac{{s}^2+8}{\sqrt{3}}\sqrt{4{s}^2-4}=\frac{2({s}^2+8)}{\sqrt{3}}
\sqrt{{s}^2-1}.$$
Hence
 $$
 \left[ \frac{m_2-l}{\sqrt{3}}\sqrt{m_2+2l}\right]^2 -m_3^2=\frac34 ({s}^2-4)^3>0.$$
We now can conclude that if all the eigenvalues are real, then $S\cap \Gamma=S\cap
\Gamma_3,$ and $\Lambda_0$ must be in the form
$r(1+s,1-s -1,-1,)$ (plus arbitrary permutation), where $-2<s<2$ and $r>0$.

\end{document}